\documentclass[10pt, article]{amsart}

\usepackage{tikz}
\usetikzlibrary{shadings}
\usepackage{pgfplots}
\usetikzlibrary{calc}
\usetikzlibrary{arrows}

\usepackage{ae} 
\usepackage[T1]{fontenc}
\usepackage[cp1250]{inputenc}
\usepackage{amsmath}
\usepackage{amssymb, amsfonts,amscd,verbatim}

\usepackage{indentfirst}
\usepackage{latexsym}
\input xy
\xyoption{all}

\usepackage{amsmath}    

\theoremstyle{plain}
\newtheorem{Prop}{Proposition}[section]
\newtheorem{Proposition}[Prop]{Proposition}

\newtheorem{Theorem}[Prop]{Theorem}

\newtheorem{Corollary}[Prop]{Corollary}

\theoremstyle{definition}
\newtheorem{Def}[Prop]{Definition}
\newtheorem{Definition}[Prop]{Definition}

\theoremstyle{remark}
\newtheorem{Rem}[Prop]{Remark}
\newtheorem{Remark}[Prop]{Remark}

\newcommand{\diam}{\operatorname{Diam}\nolimits}

\newcommand{\Rips}{\operatorname{Rips}\nolimits}
\newcommand{\cRips}{\operatorname{\overline{Rips}}\nolimits}

\def\RR{{\mathbb R}}
\def\FF{{\mathbb F}}

\def\II{{\mathbb I}}

\def\NN{{\mathbb N}}
\def\P{{\mathcal P}}

\def\L{{\mathcal L}}

\def\eps{{\varepsilon}}
\def\s{\sigma}

 \def\size{\mathop{\textrm{size}}\nolimits}
  
\def\length{\mathop{\textrm{length}}\nolimits}

\def\argmin{\mathop{\textrm{argmin}}}

\numberwithin{equation}{section}
\title[Approximations of 1-Dimensional ...]
{Approximations of 1-Dimensional Intrinsic Persistence of geodesic spaces and their stability}

\author{\v Ziga ~Virk}
\address{University of Ljubljana}
\email{zigavirk@gmail.com}

\thanks{Most of the research was conducted while the author was a postdoctoral researcher in the group of prof. Edelsbrunner at IST Austria. 
Research was also partially supported by Slovenian Research Agency grant No. J1-8131.}

\begin{document}

\maketitle
\begin{center}
\today
\end{center}

\begin{abstract}
A standard way of approximating or discretizing a metric space is by taking its Rips  complexes. These approximations for all parameters are often bound together into a filtration, to which we apply the fundamental group or the first homology. We call the resulting object  persistence. 

Recent results demonstrate that  persistence of a compact geodesic locally contractible space $X$ carries a lot of geometric information. However, by definition the corresponding Rips complexes have uncountably many vertices. In this paper we show that nonetheless, the whole persistence of $X$ may be obtained by an appropriate finite sample (subset of $X$), and that persistence of any subset of $X$ is well interleaved with the persistence of $X$. It follows that the persistence of $X$ is the minimum of persistences obtained by all finite samples. Furthermore, we prove a much improved Stability theorem for such approximations. As a special case we provide for each $r>0$ a density $s>0$, so that for each $s$-dense sample  $S \subset X$ the corresponding fundamental group (and the first homology) of the Rips complex of $S$ is isomorphic to the one of $X$, leading to an improved  reconstruction result.
\end{abstract}

\section{Introduction}

Given a metric space $X$, its induced filtration $\{\Rips(X,r)\}_{r>0}$ by Rips complexes (also called Vietoris-Rips complexes, see Definition \ref{DefComplexes}) is at the heart of  modern topological data analysis (TDA) and has also appeared in other contexts, such as shape theory and coarse geometry \cite{Comb}. Such filtrations are often preferred to \v Cech filtrations due to their computational and combinatorial simplicity, although an absence of the corresponding Nerve theorem makes a relation to the underlying space a bit more complicated. It was nevertheless shown that for small parameters $r$ and nice spaces, the Rips complex captures the homotopy type of a space \cite{Haus, Lat, Att}.

In this paper we will study the induced $1$-dimensional persistence (referred to as just persistence throughout the paper) obtained by applying the fundamental group or the first homology to such filtration of a pointed compact geodesic space $(X, \bullet)$.  The core theory of such persistence was developed in \cite{ZV}, where a connection to geometric properties (lengths of geodesic) and structural properties were described. These connections play a role, similar to that  of the Nerve Theorem, in the context of Rips complexes: for each parameter $r$ they explain the geometric features captured by the construction of the corresponding Rips complex (and in the entire persistence). As such they support and develop a geometric understanding about what information do the methods of TDA actually extract.

The aim of this paper is to describe how  persistence of such a space, with the vertex sets of the underlying Rips complexes being uncountable, may be related to a persistence of a finite subset (sample). There are two reasons for establishing such a relation. First, the corresponding results explain how such persistence may be computed using finite samples, paving the way for a future implementation of the process. Second, it allows us to deduce various finiteness results about persistence, some of which we are to explore in the future work. For example, it was proved in \cite{Lat} that for each closed Riemannian manifold there exists $\eps_0 >0$ so that the Rips complex of each finite $\eps_0$-dense subset (see Section 2 paragraph 4 for the definition) is homotopy equivalent to $X$. Since it was known before \cite{Haus}, that for small parameters the Rips complex of such a manifold is homotopy equivalent to the manifold itself, the result of \cite{Lat} essentially states that the Rips complex for small parameters can, up to homotopy, be obtained by finite samples. In this paper we obtain such a result for the fundamental group (and consequently the first homology group) of Rips complexes of any geodesic space for  any parameter (Theorem \ref{ThmRealizRips}): such a group may be obtained from a sufficiently dense sample, where the required density depends on the closest critical value on the left. This, for example, implies that all such fundamental groups are finitely presented.   For the reconstructions of the fundamental group of a semi-locally simply connected space it follows from \cite{ZV} that we also get an explicit geometric bound: if $S\subset X$ is $s$-dense, where $6s$ is less than $\ell$ (with $\ell$ being the length of the shortest non-contractible loop), then $\pi_1(\Rips(S,\ell/3), \bullet) \cong \pi_1(X, \bullet)$.

The other main insights of the paper are the following:  
\begin{enumerate}
 \item Section 3:  The persitence of any $s$-dense subset of $X$ is nicely interleaved with the one of $X$, with many maps being surjective;
 \item Theorem \ref{ThmRealizPers}: For each $p>0$ the entire persistence for parameters $r>p$ may be obtained from a finite sample. Such sample is obtained by adding to a sufficiently dense sample three equidistant points onto each corresponding critical geodesic circle.
 \item Theorem \ref{ThmSamplingMin}: Persistence of $X$ is the minimum of persistences obtained by finite samples (with respect to a natural total order to be defined later).
  \item Theorem \ref{ThmStab}: Approximations by finite samples are much more stable than the standard Stability result suggest, see Figure \ref{FigStab} for a visualisation.
\end{enumerate}

The field of TDA is usually concerned with computing persistent homology. Results of this paper  show that even a persistent fundamental group of geodesic spaces, mentioned as $\pi_1$-persistence or just persistence throughout this paper, may be within reach. Papers \cite{BD} and \cite{E} contain computations of groups that may be applied to our setting and \cite{Le} provides a setting in which the persistent fundamental group was first applied.

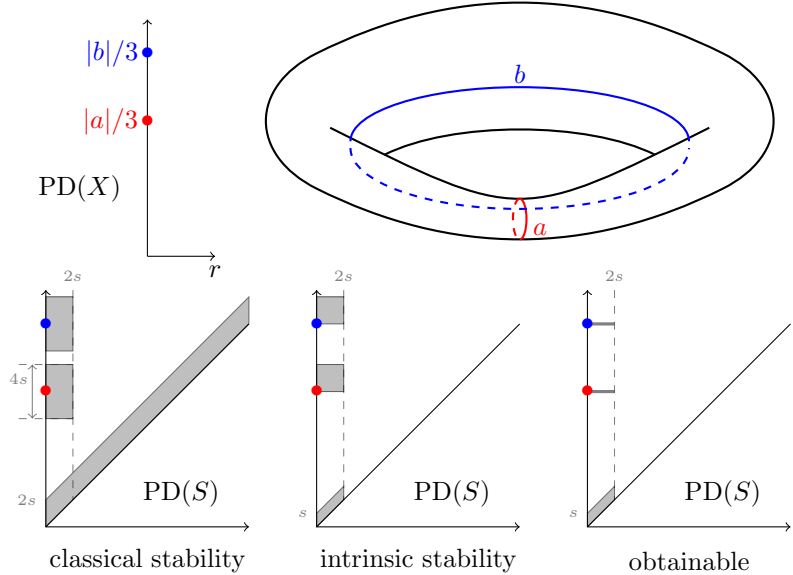
\begin{figure}
\begin{tikzpicture}[scale=.9]
\draw [thick] (-3,0) ..controls (-1,-1)and(1,-1)..
(3,0)..controls(4,.5)and(4, 1.5)..
(3,2)..controls(1,3)and(-1, 3)..
(-3,2)..controls(-4,1.5)and(-4, .5)..
cycle;
\draw [thick](-2.8,.9) ..controls (0,-.5)..(2.8,.9);
\draw  [thick](-2,.5) ..controls (-1,1)and(1,1)..(2,.5);
\draw [red, thick](0,-.75) arc (-90:90:.1 and .3);
\draw [red, thick, dashed](0,-.15) arc (90:270:.1 and .3);
\draw [red](0.3, -.6) node {$a$};
\draw [blue, thick](2.5,.7) arc (0:180:2.5 and .8);
\draw [blue, thick, dashed](-2.5,.6) arc (180:360:2.5 and .9);
\draw [blue](0, 1.7) node {$b$};
\draw[->] (-5.5, -1) to (-4.5, -1) node[below]{$r$};
\draw[->] (-5.5, -1) to (-5.5, 2.5) node[above]{};
\draw[red] (-5.5, 1) node {$\bullet$} node[left]{$|a|/3$};
\draw[blue] (-5.5, 2) node {$\bullet$}node[left]{$|b|/3$};
\draw (-6.5, 0) node{PD($X$)};

\draw (-5.5, -5.5) node{classical stability};
\draw (-1.5, -5.5) node{intrinsic stability};
\draw (2.5, -5.5) node{obtainable};
\draw[dashed, opacity=.5] (-7.4, -2.6)node[below ]{\tiny$4s$}to(-6.6, -2.6);
\draw[dashed, opacity=.5] (-7.4, -3.4)to(-6.6, -3.4);
\draw[thin, opacity=.5, <->] (-7.2, -3.4)to(-7.2, -2.6);
\draw [fill=gray, opacity=.5] (-7,-3.4)--(-7, -2.6)--(-6.6, -2.6)--(-6.6, -3.4)--cycle;
\draw [fill=gray, opacity=.5] (-7,-2.4)--(-7, -1.6)--(-6.6, -1.6)--(-6.6, -2.4)--cycle;
\draw [fill=gray, opacity=.5] (-3,-3)--(-3, -2.6)--(-2.6, -2.6)--(-2.6, -3)--cycle;
\draw [fill=gray, opacity=.5] (-3,-2)--(-3, -1.6)--(-2.6, -1.6)--(-2.6, -2)--cycle;
%
\draw [very thick, gray] (1,-3)--(1.4, -3);
\draw [very thick, gray] (1,-2)--(1.4, -2);
\draw [fill=gray, opacity=.5] (-7,-5)--(-7, -4.6)node[left]{\tiny$2s$}--(-4, -1.6)--(-4, -2)--cycle;
\draw [fill=gray, opacity=.5] (-3,-5)--(-3, -4.8)node[left]{\tiny$s$}--(-2.6, -4.4)--(-2.6, -4.6)--cycle;
\draw [fill=gray, opacity=.5] (1,-5)--(1, -4.8)node[left]{\tiny$s$}--(1.4, -4.4)--(1.4, -4.6)--cycle;
\foreach \x in {-7, -3, 1}
{
\draw[->] (\x, -5) to (\x+3, -5) node[below]{};
\draw[->] (\x, -5) to (\x, -1.5) node[above]{};
\draw (\x, -5) to (\x+3, -2);
\draw[red] (\x, -3) node {$\bullet$} node[left]{};
\draw[blue] (\x, -2) node {$\bullet$}node[left]{};
\draw (\x + 2, -4.5) node{PD($S$)};
\draw[dashed, opacity=.5] (\x+.4, -1.5)node[above]{\tiny$2s$}to(\x+.4, -4.6);
}

\end{tikzpicture}
\caption{A sketch of the stability result for an $s$-dense subset $S \subset X$. The classical stability using the Gromov-Hausdorff distance $s$ between $S$ and $X$ implies the matched points will be in the gray region (left diagram). For the intrinsic setting the gray region is significantly smaller by Theorem \ref{ThmStab}, meaning the approximation of the persistence diagram is much better (central diagram). The death times, which are the only information contained in $PD(X)$ can actually be obtained precisely due to Theorem \ref{ThmRealizPers} (the diagram on the right).}
\label{FigStab}
\end{figure}

\textbf{Overview.} In Section 2 we set up the preliminary notations and results. In Section 3 we establish the core interleaving and its properties. In Section 4 we prove that persistence can be essentially obtained by appropriate finite samples. We then use these results in Section 5 to state how the intrinsic persistence is the minimal persistence obtained by finite samples, and in Section 6 to show that such sampling is very stable. We conclude with some remarks on generalizations and future work.

\section{Preliminaries}

We will briefly recall the setting of intrinsic persistence, for geodesic spaces. For more details see \cite{ZV}. As a reference for topological notions we suggest \cite{Hat}.

Let $(X,d)$ be a geodesic metric space, i.e., a space such that for each $x,y\in X$ there exists a path, called \textbf{geodesic}, from $x$ to $y$ of length $d(x,y)$. Examples of such spaces include Riemannian manifolds and appropriately metrized complexes. A basepoint will be denoted by $\bullet \in X$. For $A\subset X$ we define the diameter as $\diam(A)=\sup_{x,y\in A}d(x,y)$. For $X \in  X$ and $r\geq 0$  the notation $B(x,r)$ represents the open ball. A \textbf{geodesic circle} in $X$ is an isometrically embedded circle $C$ of positive circumference, where the metric on $C$ as the domain of such embedding is the geodesic metric. In particular, for any two points on a geodesic circle there exists a geodesic between them that is contained in $C$.

Given a loop $\alpha$ in $X$, the induced homotopy class (based, if $\alpha$ is based) is referred to as $[\alpha]_{\pi_1}$ and the induced cycle in $H_1(X;G)$ is referred to as $[\alpha]_G$. The concatenation of loops or paths $\alpha$ and $\beta$ is denoted by $\alpha * \beta$. When concatenating paths  we naturally require that the endpoint of $\alpha$ is the starting point of $\beta$. Given a path $\alpha \colon [0,a]\to X$ the inverse path $\alpha^-\colon [0,a]\to X$  is defined by $\alpha^-(t)=\alpha(a-t)$. A \textbf{lasso} is a path of the form $\alpha * \beta * \alpha^-$, where $\alpha$ is a path and $\beta$ is a loop based at the endpoint of $\alpha$. A lasso $\alpha * \beta * \alpha^-$ is geodesic if $\beta$ is a geodesic circle. Given $r>0$ we define $\L(X,r,\pi_1)\leq \pi_1(X,\bullet)$ as the subgroup generated by all lassos $\alpha * \beta * \alpha^-$ based at $\bullet$ for which $\length(\beta)< r$. We also set $\L(X, fin, \pi_1)=\bigcup_{n\in \NN}\L(X, n, \pi_1)$.

Given $s>0$ we say $S\subset X$ is $s$-\textbf{dense} if $\forall x\in X \ \exists y\in S: d(x,y)<s$, or equivalently, if $\{B(y,s)\}_{y\in S}$ is a cover of $X$.

\begin{Def}\label{DefComplexes}
Given  $r>0$ we define the \textbf{(Open) Rips} (also called Vietoris-Rips) \textbf{complex} as a  complexes with vertex set $
X$, so that $\s\subset X$ belongs to the complex if and only if $ Diam(\s) < r$.
\end{Def}

We will call $1$-dimensional simplices edges, and $2$-dimensional simplices triangles. The following definition partially appeared in \cite{ZV}.

\begin{Def}\label{DefFiltration}
The \textbf{(open) Rips filtration} of $X$ is a collection $\{\Rips(X,r)\}_{r>0; r\in\RR}$ along with naturally induced simplicial \textbf{bonding} maps $i_{p,q}\colon \Rips(X,p)\to \Rips(X,q)$ for all $p<q$, which are identities on vertices. Considering this induced filtration as a category, we will apply various functors $\P$ to obtain persistent objects (e.g., persistent groups or persistent vector spaces). We will refer to these persistent objects as just \textbf{persistence}, when the algebraic context is clear. In this paper we will  focus on $\pi_1$-persistence and $H_1$-persistence, both of which are persistent groups. If $\FF$ is a field, then $H_1(\_; \FF)$-persistence consists of vector spaces and is usually referred to as a persistence module. An \textbf{isomorphism} $f$ between persistences $\{A_r\}_{r>0}$ and $\{B_r\}_{r>0}$ is a collection of isomorphisms $f_r \colon A_r \to B_r$ which commutes with the corresponding bonding maps. The isomorphism is denoted by $\{A_r\}_{r>0}\cong \{B_r\}_{r>0}$.

Given a persistence we say that $p>0$ is a:
\begin{enumerate}
 \item \textbf{left critical value}, if for all small enough $\eps>0$ the map $i_{p-\eps,p}$ is not an isomorphism (in the category determined by $\P$);
  \item \textbf{right critical value}, if for all  $\eps>0$, the map $i_{p,p+\eps}$ is not an isomorphism;
\end{enumerate}
We call $p>0$ a \textbf{critical value}, if it is either of the above.

Suppose now that $\FF$ is a field and that for all $r>0$,  vector spaces $H_1(\Rips(X; r); \FF)$ are finitely generated. In this case the corresponding persistence (persistence module) decomposes as a sum of elementary intervals  $\II_{\langle b,d \rangle}$, each of which is a  persistence module of the form
$$
(\II_{\langle b,d \rangle})_r = 
\begin{cases}
 \FF, \quad r\in \langle b,d \rangle \\
 0, \quad r\notin \langle b,d \rangle,
\end{cases}
$$ 
with bonding maps being surjective (identities or trivial) and with $ \langle b,d \rangle$ denoting a general interval (with either endpoint being open or closed). In \cite{ZV} it was proved that in the case of a compact geodesic space,  all intervals are open on left and closed on right.
 \textbf{Persistence diagram} $PD(X,\FF)$ is a set consisting of points $(b,d)\in \RR^2, (b<d)$, which appear as endpoints of the above mentioned intervals ($b$ is referred to as the birth, and $d$ as the death of an interval). If we want to indicate whether an interval is open or closed at a specific endpoint we may use decorated points as in \cite{Cha1}. To each point we attach a degree indicating the number of intervals with the corresponding endpoints. For a simpler statement of the stability result below we add to each $PD$ the diagonal $\{(x,x)\mid x>0\}$ in the form of an uncountable collection of points of infinite multiplicity. The \textbf{bottleneck distance} $d_B(\{A_r\}_{r>0}, \{B_r\}_{r>0})$ between persistence modules $\{A_r\}_{r>0}$ and $\{B_r\}_{r>0}$ is the minimal number $M$, for which there exists a perfect $M$-\textbf{matching} between the corresponding persistence diagrams, i.e., a bijection $\phi$ between points (where a point of multiplicity $n$ is considered to consist of $n$-copies of a single point) of persistence diagrams so that for each point $A$ we have $d_\infty((a_1, a_2), \phi(a_1, a_2))\leq M$, where $d_\infty ((x_1, y_2), (x_2, y_2))=\max \{|x_1-x_2|, |y_1-y_2|\}$.

Given $\delta>0$, a $\delta$-\textbf{interleaving} between persistences $\{A_r\}_{r>0}$ and $\{B_r\}_{r>0}$ is a collection of homomorphisms $f_r \colon A_r \to B_{r+\delta}$ and $g_r \colon B_r \to A_{r+\delta}$ which commutes with each other and with the corresponding bonding maps. The celebrated Stability Theorem  \cite{Stab1, Stab2} states the following: if persistence modules $\{A_r\}_{r>0}$ and $\{B_r\}_{r>0}$ are $\delta$-interleaved then $d_B(\{A_r\}_{r>0}, \{B_r\}_{r>0})\leq\delta$.
 \end{Def}

\begin{Def}
 \label{DefSLSC}
A path-connected space $X$ is \textbf{semi-locally simply-connected} (SLSC) if for every point $x\in X$ there exists a neighborhood $U$ of $x$ for which the image of the inclusion induced map $\pi_1(U,x)\to \pi_1(X,x)$ is trivial.

Given an Abelian group $G$, a path-connected space $X$ is $G$-\textbf{semi-locally simply-connected} ($G$-SLSC) if for every point $x\in X$ there exists a neighborhood $U$ of $x$ for which the image of the inclusion induced map $H_1(U; G)\to H_1(X; G)$ is trivial.
\end{Def}

Throughout the paper we will usually prove results for $\pi_1$-persistence. The homological analogues can be deduced in a similar manner or directly using the Hurewicz Theorem and the Universal Coefficients Theorem.

Next we present an extensive definition and several results results of \cite{ZV} that will be useful in for arguments. 

\begin{Def}\cite[excerpt from Definition 3.1]{ZV}
\label{DefRLoop}

 Fixing $r>0$ we define the following notation:
 \begin{enumerate}
 	\item $r$-\textbf{loop} $L$: a simplicial loop in $\Rips(X,r)$ considered as a sequence of points $(x_0, x_1, \ldots, x_k, x_{k+1}=x_0)$ in $X$ with $d(x_i, x_{i+1})<r, \forall i\in \{0,1,\ldots, k\}$. 
	\item \textbf{filling} of $L$: any loop in $X$ obtained from $L$ by connecting $x_i$ to $x_{i+1}$ by a geodesic for all $i\in \{0,1,\ldots, k\}$;
	\item $\size(L)=|L|=k+1$;
	\item $r$-\textbf{sample} of a loop $\alpha \colon [0,a]\to X$: a choice of $0\leq t_0<t_1<\ldots < t_m \leq a$ with $\diam \alpha ([t_{i},t_{i+1}])<r, \forall i\in \{0,1,\ldots, m-1\}$ and $\diam (\alpha([0,t_0] \cup[t_m,a]) <r$. Such choice of $t_i$'s exists by compactness. By an $r$-sample we will usually consider the induced $r$-loop $(\alpha(t_0), \alpha(t_1), \ldots, \alpha(t_m), \alpha(t_{0}))$. If $\alpha$ is based at $\bullet$, we will assume $t_0=0$;
 \end{enumerate}
 
 An $r$-loop is $r$-\textbf{null} if it is contractible in $\Rips(X,r)$. Two $r$-loops are $r$-\textbf{homotopic}, if they are homotopic in $\Rips(X,r)$. The corresponding simplicial homotopy in $\Rips(X,r)$ is referred to as $r$-\textbf{homotopy}. If the second $r$-loop is constant we also call it $r$-\textbf{nullhomotopy}. Depending on the context we may be considering based or unbased homotopies. The \textbf{concatenation} $L * L'$ of $r$-loops $L$ and $L'$ is defined in the obvious way by joining (concatenating) the defining sequences. Note that a filling of the join is the join of fillings of $r$-loops.
\end{Def}

\begin{Prop}\cite[Proposition 3.2]{ZV}
\label{PropRips}
 Let $X$ be geodesic and fix $0<r<r'$. Then the following hold:
 \begin{enumerate}
 	\item if $L$ is an $r$-loop then it is an $r'$-loop as well;
	\item if $r$-loop $L$ is $r$-null then it is $r'$-null as well;
	\item any $r$-loop of size $3$ is $r$-null;
	\item given a loop $\alpha \colon [0,a]\to X$, any two $r$-samples of $\alpha$ are $r$-homotopic;
	\item any $r$-sample of a loop of length less than $3r$ is $r$-null;
	\item choose loops $\alpha \colon [0,a]\to X$ and $\alpha' \colon [0,a']\to X$ and take any two of their $r$-samples $L$ and $L'$. If $\alpha$ and $\alpha'$ are  homotopic, then  $L$ and $L'$ are $r$-homotopic (the statement holds for both based and unbased versions). If $G$ is an Abelian group and $[\alpha]_G=[\alpha']_G\in H_1(X; G)$ then $[L]_G=[L']_G\in H_1(\Rips(X,r);G)$; 
	\item if a loop $\alpha \colon [0,a]\to X$, is contractible, then any of its $r$-samples is $r$-null;
	\item suppose two $r$-loops, $L$ and $L'$, are given by $(x_0, x_1, \ldots, x_k, x_{k+1}=x_0)$ and $(y_0, y_1, \ldots, y_k, y_{k+1}=y_0)$. If
	$$
	\max_{i\in \{0,1,\ldots, k\}}d(x_i,y_i) < r- \max_{i\in \{0,1,\ldots, k\}} \{d(x_i,x_{i+1}),d(y_i, y_{i+1})\},
	$$
	then $L$ and $L'$ are $r$-homotopic (the statement holds for both based and unbased versions);
	\item maps $\pi_1(i_{p,q})\colon \pi_1(\Rips(X,p),\bullet)\to \pi_1(\Rips(X,q), \bullet)$ and their homological counterpart are surjective for all $p<q$.
\end{enumerate}
\end{Prop}

\begin{Theorem} 
\cite[excerpt from Theorem 7.1]{ZV}\label{ThmFin}
 Suppose $X$ is compact and geodesic,  and let $q$ be a critical value for $\pi_1$ persistence of $X$ via open Rips filtration.
 \begin{description}
 \item [a] For each $r>0$ there exist only finitely many critical values for $\pi_1$ persistence via open Rips complexes, which are greater than $r$. 
\item [b] Let $q<p$ be a pair of consecutive critical values. Group $\L(X, 3p, \pi_1)$ is generated by  $\L(X, 3q, \pi_1)$ and a collection of geodesic $3q$-lassos.
 
\end{description}
 \end{Theorem}

\begin{Theorem}
\cite[Theorem 7.8]{ZV}.
\label{ThmPerBasis}
 [Persistence-basis correspondence] 
 Suppose $X$ is a compact geodesic space. Then there exist geodesic lassos $\{\beta_i=\alpha_i * \gamma_i * \alpha_i^-\}_{i\in J}$, with $\gamma_i$ being a geodesic circle of length $l_i$, whose normal closure is $\pi_1(X, \bullet)$, so that the following isomorphism holds: 
 $$\big\{\pi_1(\Rips(X,r),\bullet)\big\}_{r>0}\cong\big\{\L(X,fin, \pi_1)/
\{\beta_i\big\}_{l_i<r}
\}_{r>0},$$
with bonding maps of the right-side persistence being the natural quotient maps. The set of critical values coincides with $\{l_i/3\}_{i\in J}$. For each critical value $c$ there exist only finitely many indices $j$ for which $l_j=3c$. 

Furthermore, if $X$ is SLSC then $J$ can be chosen to be finite and
 $$\big\{\pi_1(\Rips(X,r),\bullet)\big\}_{r>0}\cong\big\{\pi_1(X, \bullet)/
\{\beta_i\big\}_{l_i<r}
\}_{r>0}.$$
\end{Theorem}

It will also be handy to know what maps induce the isomorphisms in Theorem \ref{ThmPerBasis}. The isomorphism  $\pi_1(\Rips(X,r),\bullet) \to \L(X,fin, \pi_1)/
\{\beta_i\big\}_{l_i<r}$ is obtained by taking any filling of an $r$-loop in $\Rips(X,r)$. In the other direction we take any $r$-sample. For details see \cite{ZV}.

The following is a generalization of Proposition \ref{PropRips}(5). For reasons of simplicity it is stated for geodesic circles $\alpha$ but actually holds for general loops of finite length.

\begin{Proposition} \label{PropSNull}
Suppose $S\subset X$ is  $s$-dense, $\gamma$ is a geodesic circle in $X$ of  length $a=\length(\gamma)$, and $r, \delta>0$ are positive reals with $ r > \delta + 2s$. 
Let $L$ be a $\delta$-sample of $\gamma$ given by $(x_0, x_1, \ldots, x_k, x_{k+1}=x_0)$. Also choose an $r$-loop $L_S$ in $X$ given by $(y_0, y_1, \ldots, y_k, y_{k+1}=y_0)$ where $d(x_i, y_i)<s$ and  $y_i\in S, \forall i$. Note that $L_S$ also represents a simplicial loop in $\Rips(S, r)$. Then the following hold:
\begin{enumerate}
 \item if $a< 3(r-2s)$, then $L_S$ is contractible in $\Rips(S, r)$;
 \item if $a< 3r$ and $S$ contains three equidistant points on $\alpha$, then $L_S$ is contractible in $\Rips(S, r)$.
\end{enumerate}
\end{Proposition}

\begin{proof}
(1) Pick three equidistant points $w_1, w_2, w_3$ on $\gamma$. For each of them choose some corresponding point in $S$, i.e., for each $i$ choose $q_i\in S, d(w_i, q_i)<s$. Note that $d(q_i, q_j)< a/3 + 2s < r$, hence $q_1, q_2, q_3$ form a simplex (the shaded triangle on Figure \ref{FigNullhomotopy}) in $\Rips(S, r)$. Furthermore, we add the following triangles by the following procedure:
\begin{itemize}
\item choose a vertex $y_{i_3}$ of $L_S$ closest to the midpoint of $w_1$ and $w_2$ along $\gamma$. If two points satisfy this condition choose any of them. Add triangle $y_{i_3}, q_1, q_2$ and note that $d(q_1, q_2) <r$, hence the triangle is contained in $\Rips(S, r)$. Repeat the procedure for the other two combinations of indices to obtain $y_{i_1}$ and $y_{i_2}$ and two more triangles contained in $\Rips(S, r)$. On Figure \ref{FigNullhomotopy} these appear as the non-shaded triangles, sharing a side with the shaded triangle.
 \item for each $i\in \{0, 1, \ldots, k\}$ define $j_i\in \{1, 2, 3\}$ as $j_i = \argmin_j d(w_j, x_i)$. If two indices $j_i$ could be chosen then $x_i$ is the midpoint between two points of the form $w_*$ and we choose $w_{j_i}$ to be any of them. Then add:	
		\begin{itemize}
 			\item for each pair $i, i+1$ with $j_i = j_{i+1}$ the triangle $y_i, y_{i+1}, q_{j_i}$;
			\item for each $j\in \{1, 2, 3\}$ triangles $y_{i_j}, y_{i_j+1}, q_{i_j+1}$ and $y_{i_j}, y_{i_j-1}, q_{i_j-1}$ (for each $j$ precisely one of these two has been added in the previous step), where we consider indices modulo $k$.
		\end{itemize}
 \end{itemize}
 These triangles together form a simplicial nullhomotopy of $L_S$ in $\Rips(S, r)$ as sketched on Figure \ref{FigNullhomotopy}.
 
\begin{figure}
\begin{tikzpicture}[scale=1]

\foreach \x  in {0, 1, ..., 11}
	\draw (-\x * 30+90:1.9) -- (-\x * 30+90:2.0) (-\x * 30+90:2.4) node {$y_{\x }$};
\foreach \x  in {0, 1, ..., 11}
	\draw [fill=black] (-\x * 30+90:1.9) circle (.1cm);
\foreach \x  in {0, 1, ..., 11}
	\draw (-\x * 30+90:1.9) -- (-\x * 30+120:1.9);
	\draw[top color=gray,bottom color=white, fill opacity =1] (1,0) -- (-.5, .7) -- (-.5,-.7) -- cycle;
\node (q1) at (1,0){$\bullet$}; 

\node (q2) at (-.5,.7){$\bullet$};

\node (q3) at (-.5,-.7){$\bullet$};

\foreach \x  in { 2, 3,..., 6}
	\draw (q1) -- (-\x * 30+120:1.9);
\foreach \x  in { 6,7, ..., 10}
	\draw (q3) -- (-\x * 30+120:1.9);
\foreach \x  in {10, 11, 12, 1, 2}
	\draw (q2) -- (-\x * 30+120:1.9);
	\node at (-.3, -.4){$q_3$};
	\node at (-.3, .4){$q_2$};
	\node at (.8, .3){$q_1$};

\end{tikzpicture}
\caption{Sketch of proof of Proposition \ref{PropSNull} (1).}
\label{FigNullhomotopy}
\end{figure}
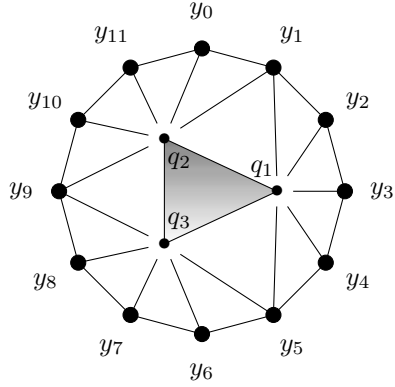
 
(2) Follows along the same lines as (1) using $w_i=g_1, \forall i\in \{1, 2, 3\}$.
\end{proof}

\section{The interleaving}

\begin{Definition}\label{DefMapsFinApprox}
Suppose $X$ is a  geodesic space, $G$ is a Abelian group,  $S\subset X$ is an $s$-dense subset containing $\bullet$ and $0<p<q$. The metric we use on $S$ is the restriction of the (geodesic) metric on $X$. The maps
$$
i^{\pi_1, S}_{p,q}\colon \pi_1(\Rips(S, p), \bullet) \to \pi_1(\Rips(S, q), \bullet), 
$$
$$
 i^{G, S}_{p,q}\colon H_1(\Rips(S, p); G) \to H_1(\Rips(S, q); G),
$$
$$
\mu^{\pi_1, S}_p \colon \pi_1(\Rips(S, p), \bullet) \to \pi_1(\Rips(X, p), \bullet),
$$ 
$$
\mu^{G, S}_p \colon H_1(\Rips(S, p); G) \to H_1(\Rips(X, p); G)
$$ 
are induced by the corresponding inclusions $\Rips(S, p)\hookrightarrow \Rips(S, q)$ and $S \hookrightarrow X$.
Maps
$$
\nu^{\pi_1, S}_p \colon \pi_1(\Rips(X, p), \bullet) \to \pi_1(\Rips(S, p+2s), \bullet),
$$ 
$$
\nu^{G, S}_p \colon H_1(\Rips(X, p); G) \to H_1(\Rips(S, p+2s); G)
$$ 
are defined by the following procedure: 
\begin{enumerate}
 \item choose a based $p$-loop $L$, given by  $\bullet = x_0, x_1, \ldots, x_k, x_{k+1}=\bullet$, representing an element of $\pi_1(\Rips(X, p), \bullet)$ or of $H_1(\Rips(X, p); G)$;
 \item for each $i$ choose $y_i\in S$ so that $d(x_i, y_i)<s$ and $y_0=y_{k+1}=\bullet$. Note that in this case we have $d(y_i, y_{i+1})<p + 2s, \forall i$ by the triangle inequality;
 \item $\nu^{*, S}_p$ maps $[L]_*$ to $[L_S]_*$, where $L_S$ is a $(p+2s)$-loop  represented by $\bullet = y_0, y_1, \ldots, y_k, y_{k+1}=\bullet$. \end{enumerate}
Technical details of the mentioned maps are explained by Proposition \ref{PropMapsFinApprox} and its proof. Their diagram is provided in the statement of Proposition \ref{PropDiagComm}.
\end{Definition}

\begin{figure}
$$
\xymatrix{
\pi_1(\Rips(S, p), \bullet) 
\ar[rr]^{i_{p, p+2s}^{\pi_1, S}}
\ar[d]^{\mu_{p}^{\pi_1, S}}
&& 
\pi_1(\Rips(S, p+2s), \bullet) 
\ar[d]^{\mu_{p+2s}^{\pi_1, S}}\\
\pi_1(\Rips(X, p), \bullet) 
\ar[rr]_{i_{p, p+2s}^{\pi_1, X}} 
\ar[urr]^{\nu_p^{\pi_1, S}}
&&
\pi_1(\Rips(X, p+2s), \bullet) 
}
 $$
 $$
\xymatrix{
\pi_1(\Rips(S, p), \bullet) 
\ar[rr]^{i_{p, q}^{\pi_1, S}}
\ar[d]^{\mu_{p}^{\pi_1, S}}
&& 
\pi_1(\Rips(S, q), \bullet) 
\ar[d]^{\mu_{q}^{\pi_1, S}}\\
\pi_1(\Rips(X, p), \bullet) 
\ar[rr]^{i_{p, q}^{\pi_1, X}} 
&&
\pi_1(\Rips(X, q), \bullet) 
}
 $$

\caption{Diagrams of maps introduced by Definition \ref{DefMapsFinApprox}.}
\label{FigDiag}
\end{figure}

\begin{Proposition}\label{PropMapsFinApprox}
 All maps introduced by Definition \ref{DefMapsFinApprox} are well defined  homomorphisms.
\end{Proposition}

\begin{proof}
All mentioned maps are obviously homomorphisms if well defined. The maps $i^{*, S}_{p,q}$ and $\mu^{*, S}_{p,q}$ are well defined as they are induced by inclusions. 

It remains to consider $\nu^{*, S}_p$. We will first prove that $\nu^{\pi_1, S}_p$ is well defined. Consider the notation of Definition \ref{DefMapsFinApprox}. We need to prove that given an element of $\pi_1(\Rips(X,p),\bullet)$, its image does not depend on a choice of $L$ and $L_S$:
\begin{itemize}
 \item independence from the choice of $L_S$: suppose $L$ is fixed and choose $(p+2s)$-loops  $L_S$ and $L_S'$ represented by $\bullet = y_0, y_1, \ldots, y_k, y_{k+1}=\bullet$ and $\bullet = y'_0, y'_1, \ldots, y'_k, y'_{k+1}=\bullet$ respectively. Then $L_S$ and $L'_S$ are $(p+2s)$-homotopic in $\pi_1(\Rips(S,p+2s),\bullet)$ with the homotopy consisting of triangles $(y_i, y'_i, y'_{i+1})$ and $(y_i, y_{i+1}, y'_{i+1})$ of diameter $p+2s$, as depicted on Figure \ref{FigPropMapsFinApprox1}. Thus $[L_S]_{\pi_1}=[L'_S]_{\pi_1}\in \pi_1(\Rips(S,p+2s),\bullet)$.
 \item independence from the choice of $L$: Suppose $L$ and $L'$, given by  $\bullet = x_0, x_1, \ldots, x_k, x_{k+1}=\bullet$ and $\bullet = x'_0, x'_1, \ldots, x'_k, x'_{k+1}=\bullet$  respectively, are $p$-homotopic. That means there is a triangulation $\Delta$ of $S^1 \times I$ and a simplicial map $H\colon  \Delta\to \Rips(X, p)$ corresponding to $L$ and $L'$ on $S^1 \times \{0\}$ and $S^1 \times \{1\}$ respectively. Let $\{v_i\}_{i\in J}$ denote the vertex set of $\Delta$. For each $i\in J$, the later being a finite set, choose $s_i\in S$ with $d(H(v_i), s_i)<s$. This assignment determines $L_S$, and $L'_S$. Also,  a $(p+2s)$-homotopy $H_S\colon (S^1\times I, \Delta)\to \Rips(S, p+2s)$ between them in $\Rips(S, p+2s)$ is induced by $H$ and $\Delta$: the combinatorial structure of the  triangulation  is the same, we only change, for each triangle $(v_{i_1},v_{i_3},v_{i_3})\in \Delta$,  triangle $H(v_{i_1},v_{i_3},v_{i_3})$ by $(s_{i_1},s_{i_3},s_{i_3})$, which is a (potentially singular) triangle in $ \Rips(S, p+2s)$ by the triangle inequality.
 \end{itemize}
  Thus $\nu^{\pi_1, S}_p$ is well defined. The homological case can be proved analogously.
 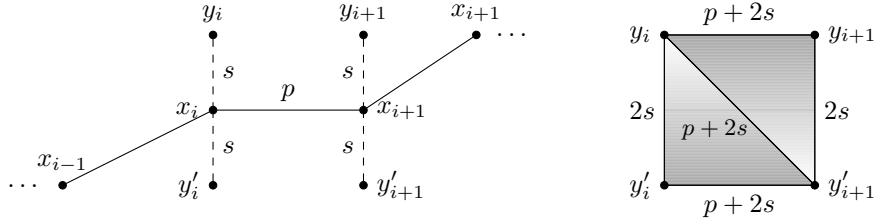
\begin{figure}
\begin{tikzpicture}
\node [circle,  scale=.3, fill=black,draw, label=left:$x_{i}$] (a1) at (0,1) {};
\node [circle,  scale=.3, fill=black,draw, label=above:$x_{i-1}$] (a0) at (-2,0) {};
\node [circle,  scale=.3, fill=black,draw, label=right:$x_{i+1}$] (a2) at (2,1) {};
\node [circle,  scale=.3, fill=black,draw, label=above:$x_{i+1}$] (a3) at (3.5,2) {};
\draw (a0) to (a1) to node [above]{$p$} (a2) to (a3);
\node [circle,  scale=.3, fill=black,draw, label=above:$y_{i}$] (b1) at (0,2) {};
\node [circle,  scale=.3, fill=black,draw, label=above:$y_{i+1}$] (b2) at (2,2) {};
%
\node [circle,  scale=.3, fill=black,draw, label=left:$y'_{i}$] (c1) at (0,0) {};
\node [circle,  scale=.3, fill=black,draw, label=right:$y'_{i+1}$] (c2) at (2,0) {};
%
\draw [dashed] (b1)to node [right]{$s$}(a1);
\draw [dashed] (c1)to node [right]{$s$}(a1);
\draw [dashed] (b2)to node [left]{$s$}(a2);
\draw [dashed] (c2)to node [left]{$s$}(a2);
%
%
\node (a5) at (-2.5,0) {$\dots$};
\node (a6) at (4,2) {$\dots$};
%
\draw[top color=gray,bottom color=white, fill opacity =.5] (8,0) -- (8,2) -- (6,2) -- cycle; 
\draw[bottom color=gray,top color=white, fill opacity =.5] (8,0) -- (6,0) -- (6,2) -- cycle; 
\node [circle,  scale=.3, fill=black,draw, label=left:$y_{i}$] (rb1) at (6,2) {};
\node [circle,  scale=.3, fill=black,draw, label=right:$y_{i+1}$] (rb2) at (8,2) {};
\draw  (rb1)to node [above]{$p+2s$} (rb2);
\node [circle,  scale=.3, fill=black,draw, label=left:$y'_{i}$] (rc1) at (6,0) {};
\node [circle,  scale=.3, fill=black,draw, label=right:$y'_{i+1}$] (rc2) at (8,0) {};
\draw  (rc1)to node [below]{$p+2s$} (rc2);
\draw  (rc1)to node [left]{$2s$}(rb1);
\draw  (rc2)to node [right]{$2s$}(rb2);
\draw  (rb1)to node [below ]{\small{$ p+2s \qquad$}}(rc2);
\end{tikzpicture}
\caption{An excerpt from the proof of Proposition \ref{PropMapsFinApprox} showing that the definition of maps $\nu_p^{*,S}$ is independent of the choice of $L_S$. Labels besides edges suggest that the distance between endpoints is less than the label. We thus obtain a $(p+2s)$-homotopy in $\Rips(S,p+2s)$ consisting of squares on the right, each of which consists of two triangles, whose endpoints are at distance less than $p+2s$.
A similar argument was given by the proof of  Proposition \cite[Proposition 3.2]{ZV}. }
\label{FigPropMapsFinApprox1}
\end{figure}
\end{proof}

\begin{Proposition}\label{PropDiagComm}
 Suppose $X$ is a  geodesic space, $G$ is a Abelian group, $0<p<q$  and $S\subset X$ is an $s$-dense subset containing $\bullet$. Then the diagrams of Figure \ref{FigDiag} and their homological counterparts commute. In particular, the diagrams provide a $(0,2s)$-interleaving between $ \{ \pi_1(\Rips(S,r), \bullet) \}_{r>0}$ and $ \{ \pi_1(\Rips(X,r), \bullet) \}_{r>0}$.
 \end{Proposition}

\begin{proof}
The upper triangle of the first diagram commutes as both  maps   can be chosen to be identities on the vertices of representing $p$-loops.

As for the lower triangle, the map $i_{p, p+2s}^{\pi_1, S}$ is the identity on the vertices of  representing $p$-loops. The other map (the composition) sends each point of a representing $p$-loop to some point at distance $s$ according to Definition \ref{DefMapsFinApprox}. However, the homotopy class of the $p$-loop does not change, which can be seen using the argument presented by Figure \ref{FigPropMapsFinApprox1}.

It is apparent that the lower diagram commutes as it is induced by inclusions. Homological versions can be proved in the same way.
\end{proof}

\subsection{Stability}

\begin{Definition}\label{DefBirthDeath}
 For a point $A$ on a $PD(Y, \FF)$ let $\mathbf{b}(A)$ denote its birth, let $\mathbf{d}(A)$ denote its death. Let $\Delta$ denote all the diagonal points.
 
  \end{Definition}

The following is the main stability result for finite subsamples. It is also sketched on Figure  \ref{FigStab}.

\begin{Theorem}\label{ThmStab} [Intrinsic Stability Theorem for finite approximations]
Suppose $X$ is a compact geodesic space, $\FF$ is a field, 
 and $S$  is an $s$-dense sample. Then there exists a perfect $2s$-matching $\phi\colon PD(X, \FF)\to PD(S, \FF)$ and the following holds:
 \begin{enumerate}
 \item $\mathbf{b}_S(A)\leq 2s, \forall A\in PD(S, \FF)\setminus \Delta$;
 \item for each $A \in PD(X,\FF)$ we have $\mathbf{d}(A) \leq \mathbf{d}(\phi(A)) \leq \mathbf{d}(A) + 2s$;
 \item for each $B \in PD(S,\FF)$ mathched to a point on $\Delta$, we have $\mathbf{d}(B) - \mathbf{b}(B) \leq s$.
\end{enumerate}
\end{Theorem}

\begin{proof}
The theorem is an immediate corollary of Proposition \ref{PropDiagComm}.
\end{proof}

%

\subsection{Surjectivity}

\begin{Proposition}\label{PropMapsFinApprox1}
The following maps of Definition \ref{DefMapsFinApprox} are surjective:
\begin{enumerate}
 \item $\mu^{*,S}_p$ for $2s\leq p$;
 \item $i^{*,*}_{p,q}$ for $2s\leq p < q$;
 \item $\nu^{*,S}_p$ for $0<p$.
\end{enumerate}
\end{Proposition}

\begin{proof}
 Note that there exists $s'<s$ so that $S$ is also an $s'$-dense sample of $X$. This is equivalent to the existence of a positive Lebesgue number of $X$ for cover  $\{B(C, s)\}_{C\in S}$. Define $\delta=s-s'$.

(1) We first prove that $\mu^{\pi_1, S}_{p}$ is surjective. This part  also follows from the commutativity of the first diagram in Proposition \ref{PropDiagComm} using Proposition \ref{PropRips} (9). However, the proof here demonstrates the necessity of the requirement $2s\leq p$ and provides the idea for the rest of this proof.  Take a based $p$-loop $L$ in $X$  representing some element in $\pi_1(\Rips(X, p), \bullet)$. Replacing $L$ by a $(p-2s)$-sample of some filling of $L$ and using Proposition \ref{PropRips} (4) we may assume that $L$ is actually a $(p-2s)$-loop represented by $\bullet = x_0, x_1, \ldots, x_k, x_{k+1}=\bullet$. For each $i$ choose $y_i\in S, d(x_i, y_i)<s$ so that  $y_0=y_{k+1}=\bullet$. Let $L'$ denote a $p$-loop represented by $\bullet = y_0, y_1, \ldots, y_k, y_{k+1}=\bullet$. On one hand it is (based) $p$-homotopic to $L$, which can be seen from the example depicted in Figure \ref{FigOpenRips}  (as $d(y_i, y_{i+1})<p$ and $d(y_i, x_{i+1})<p-s,  \forall i$ by the triangle inequality). On the other hand, $L'$ represents an element of $\pi_1(\Rips(S, p), \bullet)$, hence $\mu^{\pi_1, S}_{p}$ is surjective. An analogous argument or an application of the Hurewicz theorem imply that each $i^{G, S}_{p, q}$ is also surjective. 

 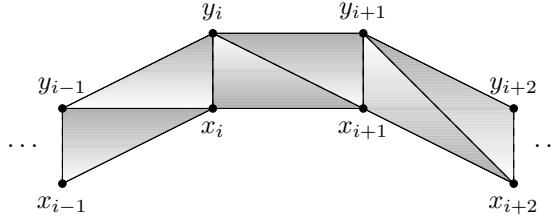
\begin{figure}
\begin{tikzpicture}
\draw[top color=gray,bottom color=white, fill opacity =.5] (-2,1) -- (0,2) --(0,1) -- cycle; 
\draw[top color=gray,bottom color=white, fill opacity =.5] (-2,1) -- (-2,0) --(0,1) -- cycle; 
\draw[bottom color=gray,top color=white, fill opacity =.5] (2,1) -- (0,2) --(0,1) -- cycle; 
\draw[top color=gray,bottom color=white, fill opacity =.5] (2,1) -- (0,2) --(2,2) -- cycle; 
\draw[bottom color=gray,top color=white, fill opacity =.5] (2,1) -- (4,0) --(2,2) -- cycle; 
\draw[top color=gray,bottom color=white, fill opacity =.5] (4,1) -- (4,0) --(2,2) -- cycle; 
\node [circle,  scale=.3, fill=black,draw, label=above:$y_{i-1}$] (a1) at (-2,1) {};
\node [circle,  scale=.3, fill=black,draw, label=above:$y_{i}$] (a2) at (0,2) {};
\node [circle,  scale=.3, fill=black,draw, label=above:$y_{i+1}$] (a3) at (2,2) {};
\node [circle,  scale=.3, fill=black,draw, label=above:$y_{i+2}$] (a4) at (4,1) {};
\draw (a1)--(a2)--(a3)--(a4);
\node [circle,  scale=.3, fill=black,draw, label=below:$x_{i-1}$] (b1) at (-2,0) {};
\node [circle,  scale=.3, fill=black,draw, label=below:$x_{i}$] (b2) at (0,1) {};
\node [circle,  scale=.3, fill=black,draw, label=below:$x_{i+1}$] (b3) at (2,1) {};
\node [circle,  scale=.3, fill=black,draw, label=below:$x_{i+2}$] (b4) at (4,0) {};
\draw (b1)--(b2)--(b3)--(b4);
\draw [dashed] (a1)--(b1);
\draw [dashed] (a2)--(b2);
\draw [dashed] (a3)--(b3);
\draw [dashed] (a4)--(b4);
\draw [dashed] (a1)--(b2);
\draw [dashed] (a2)--(b3);
\draw [dashed] (a3)--(b4);
\node (a5) at (-2.5,.5) {$\dots$};
\node (a6) at (4.5,.5) {$\dots$};
\end{tikzpicture}
\caption{An excerpt of $p$-homotopy from (1) of Proposition \ref{PropMapsFinApprox1}.}
\label{FigOpenRips}
\end{figure}

(2) Maps $i^{\pi_1, X}_{p,q}$ are surjective by Proposition \ref{PropRips}(9). We now prove that $i^{\pi_1, S}_{p,q}$ is surjective. Take a based $q$-loop $L$ in $\Rips(S,q)$ representing some element in $\pi_1(\Rips(S, q), \bullet)$. Suppose $L$ is represented by $\bullet = x_0, x_1, \ldots, x_k, x_{k+1}=\bullet$. For each $i$ choose a geodesic $\alpha_i$ from $x_i$ to $x_{i+1}$ in $X$. Note that a concatenation of paths $\alpha_i$ in the natural order is a filling of $L$. Let us focus on a single $i$ for a moment, to avoid a double indexation. Take a $\delta$-sample of $\alpha_i$ containing the endpoints and the midpoint $a_M$ of $\alpha_i$: $x_i, a_1, a_2, \ldots, a_M, \ldots, a_{m}, x_{i+1}$. For each $a_j$ choose $b_j\in S$ so that $d(a_j, b_j)<s'$. Note that $x_i, b_1, b_2, \ldots, b_{m}, x_{i+1}$ is a $2s$-loop as $\delta + 2 s' < 2s$ and hence represented in $\Rips(S, p)$ as $2s<p$. It is also $q$-homotopic to the $q$-path $x_i, x_{i+1}$ in $\Rips(S, q)$ with fixed endpoints as can be seen by combining triangles $(x_i, b_j, b_{j+1})$ for $j<M$, $(x_{i+1}, b_j, b_{j+1})$ for $j\geq M$, and $(x_i, x_j, x_M)$ as depicted on Figure \ref{FigOpenRips}. Therefore we may concatenate these obtained $p$-paths for all $i$ (in the obvious way) to obtain a $p$-loop in  $\Rips(S, p)$, which is $q$-homotopic to $L$ in $\Rips(S, q)$.  Thus $i^{\pi_1, S}_{p,q}$ is surjective.
 \begin{figure}
\begin{tikzpicture}
\node [circle,  scale=.3, fill=black,draw, label=below:$x_{i}$] (a1) at (-3,0) {};
\node [circle,  scale=.3, fill=black,draw, label=below:$x_{i+1}$] (a2) at (2,0) {};
\draw[dashed] (a1) to  node[above] {$q$}(a2);
\draw (a1) ..controls (0,-1)..  (a2);
\node [circle,  scale=.3, fill=black,draw, label=below:$a_1$] (a3) at (-2,-.34) {};
\node [circle,  scale=.3, fill=black,draw, label=below:$a_2$] (a4) at (-1,-.62) {};
\node [rectangle,  scale=.5, fill=black,draw, label=below:$a_3$] (a5) at (0,-.76) {};
\node [circle,  scale=.3, fill=black,draw, label=below:$a_4$] (a6) at (1,-.49) {};
\node [circle,  scale=.3, fill=black,draw, label=above:$b_1$] (b3) at (-2,2) {};
\node [circle,  scale=.3, fill=black,draw, label=above:$b_2$] (b4) at (-1,2.5) {};
\node [rectangle,  scale=.5, fill=black,draw, label=above:$b_3$] (b5) at (0,2.5) {};
\node [circle,  scale=.3, fill=black,draw, label=above:$b_4$] (b6) at (1,2.5) {};
\draw[dotted] (a3) to  node[right] {$s'$}(b3);
\draw[dotted] (a4) to  node[right] {$s'$}(b4);
\draw[dotted] (a5) to  node[right] {$s'$}(b5);
\draw[dotted] (a6) to  node[right] {$s'$}(b6);
%
%
\node [circle,  scale=.3, fill=black,draw, label=below:$x_{i}$] (ra1) at (3,0) {};
\node [circle,  scale=.3, fill=black,draw, label=below:$x_{i+1}$] (ra2) at (8,0) {};
\draw[top color=gray,bottom color=white, fill opacity =.5] (ra1) -- (4,2) -- (5,2.5) -- cycle; 
\draw[top color=gray,bottom color=white, fill opacity =.5] (ra1) -- (5,2.5) -- (6,2.5) -- cycle; 
\draw[top color=gray,bottom color=white, fill opacity =.5] (ra2) -- (7,2.5) -- (6,2.5) -- cycle; 
\draw[top color=white,bottom color=gray, fill opacity =.5] (3,0) -- (8,0) -- (6,2.5) -- cycle; 
\node [circle,  scale=.3, fill=black,draw, label=above:$b_1$] (rb3) at (4,2) {};
\node [circle,  scale=.3, fill=black,draw, label=above:$b_2$] (rb4) at (5,2.5) {};
\node [rectangle,  scale=.5, fill=black,draw, label=above:$b_3$] (rb5) at (6,2.5) {};
\node [circle,  scale=.3, fill=black,draw, label=above:$b_4$] (rb6) at (7,2.5) {};

\end{tikzpicture}
\caption{An excerpt from the proof of Proposition \ref{PropMapsFinApprox1}(2). The labels of connections suggest that the corresponding distance is less than the label. For a fixed $i$ we connect $x_i$ to $x_{i+1}$ by a geodesic $\alpha_i$ (solid line on the left).  We take a $\delta$-sample of $\alpha_i$ in the form of $x_i, a_1, a_2, a_3, a_4, x_{i+1}$ where $a_3=a_M$ is the midpoint of $\alpha_i$, denoted by a square. For each of $a_j$ we choose $b_j$. The triangulation on the right side provides a $q$-homotopy in $\Rips(S, q)$ between $q$-paths $x_i, x_{i+1}$ and $x_i, b_1, b_2, b_3, b_4, x_{i+1}$ with fixed endpoints.}
\label{FigPropMapsFinApprox2}
\end{figure}
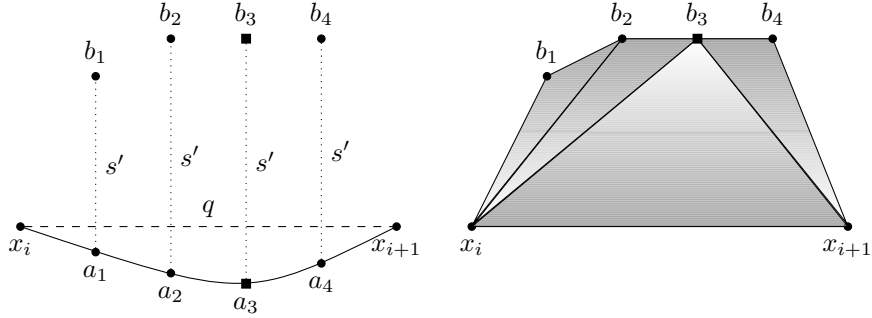
As before, an analogous argument or an application of the Hurewicz theorem imply that each $\mu^{G, S}_{p}$ is also surjective.

(3) It remains to check that maps $\nu^{*, S}_p$ are surjective. Again we will only consider $\nu^{\pi_1, S}_p$. Take a based $(p+2s)$-loop $L_S$  in $\Rips(S, p+2s)$  represented by $\bullet = y_0, y_1, \ldots, y_k, y_{k+1}=\bullet$ and choose its filling $\alpha$. Let $L$ denote a based $p$-sample of  $\alpha$, so that no point of $L$ lies on the midpoint of any of the filled geodesic segment between consecutive vertices of $L_S$, i.e., no vertex of $L$ lies on the midpoint of $\alpha_{i}$, where the latter is the geodesic segment between $y_i$ and $y_{i+1}$ defining $\alpha$. Map $L$ by $\nu^{\pi_1, S}_p$ so that each vertex on $\alpha_i$ is mapped to the closer endpoint of $\alpha_i$, which means either $y_i$ or $y_{i+1}$, according to  Definition \ref{DefMapsFinApprox} (2). This implies that $L$ is mapped to $L_S$ with repetitions of points, for example $\bullet = y_0, y_0, y_1, y_1, \ldots, y_1, y_2, \ldots, y_k, y_k, y_{k+1}=\bullet$. The later is obviously $(p-2s)$-homotopic to $L_S$ in $\Rips(S, p+2s)$.
\end{proof}

\begin{Corollary}
Consider the notation of Definition \ref{DefMapsFinApprox}. Suppose $\tilde s$ is the supremum of Lebesgue numbers of  the cover  $\{B(C, s)\}_{C\in S}$. Then the following maps of Definition \ref{DefMapsFinApprox} are surjective:
\begin{enumerate}
 \item $i^{*,*}_{p,q}$ for $s-\tilde s < p < q$;
 \item $\mu^{*,*}_p$ for $2 (s-\tilde s)< p$.
\end{enumerate}
\end{Corollary}

\begin{proof}
 Note that $S$ is $s'$ dense for each $s'> s - \tilde s$ and use Proposition \ref{PropMapsFinApprox1}.
\end{proof}

\section{Approximation}

\begin{Proposition}\label{PropIso}
 Suppose $X$ is a  geodesic space, $G$ is an Abelian group, $0<p$ and $S\subset X$ is an $s$-dense subset containing $\bullet$. Suppose $q\geq 0$ and there are no critical values for $\pi_1$-persistence (or for $H_1(\_; G)$-persistence) of open Rips complexes of $X$ on $(p, p+2s+q]$. Then:
 \begin{enumerate}
 \item maps $\mu_r^{\pi_1, S}$ ($\mu_r^{G, S}$ respectively) are isomorphisms for all $r\in (p+2s, p+2s+q]$.
 \item maps $i_{p', q'}^{\pi_1, S}$ ($i_{p', q'}^{G, S}$ respectively) are isomorphisms for all $[p', q']\subseteq (p+2s, p+2s+q]$.
\end{enumerate}
\end{Proposition}

\begin{proof}
We will only provide a proof for the $\pi_1$ case. Since $\mu_r^{\pi_1, S}$ is surjective by Proposition \ref{PropMapsFinApprox}, we only need to prove it is injective. Since map $\nu_{r-2s}^{\pi_1, S}$ is surjective (by Proposition \ref{PropMapsFinApprox}) and $i_{r-2s, r}^{\pi_1, X}$ is an isomorphism (by the assumption on the critical values), the commutativity of the lower triangle of the upper diagram of Proposition \ref{PropDiagComm} implies $\mu_r^{\pi_1, S}$ is injective, hence an isomorphism. The last statement of the proposition holds by the commutativity of the lower diagram in Proposition \ref{PropDiagComm}.
\end{proof}

\begin{Theorem}\label{ThmRealizRips}
[Obtaining fundamental groups of Rips complexes by finite samples]
Suppose $X$ is a compact geodesic space, $G$ is a Abelian group and $r>0$. Let $c$ be the largest critical value of  $\pi_1$-persistence (or $H_1(\_; G)$-persistence) of $X$ via open Rips complexes, that is smaller than $r$ (if no critical value is smaller than $r$ we set $c$ to be any number on $(0,r)$). Such $c$ exists by Theorem \ref{ThmFin} \textbf{a}.
 Choose a finite $(r-c)/2$-dense sample $S\subset X$ containing $\bullet$. Then $\mu^{\pi_1, S}_r \colon \pi_1(\Rips(S, r), \bullet) \to \pi_1(\Rips(X, r), \bullet)$ (or $\mu^{G, S}_r$ respectively) is an isomorphism.
\end{Theorem}

\begin{proof}
 Follows from Proposition \ref{PropIso}.
 \end{proof}

\begin{Remark}
 As a corollary we see the following: given a compact geodesic SLSC space, its fundamental group is finitely presented. Using an existence of a geodesic metric (via compactness and convex metric as in \cite{DV}) we can generalize this result to compact, connected, locally connected SLSC metric spaces. This is the main result of \cite{DV} and also appears in \cite[Lemma 7.7]{CC}.  
\end{Remark}

\begin{Remark}\label{RemException}
A version of Theorem \ref{ThmRealizRips} in the context of closed Rips filtrations  holds for all values $r$, that are not a critical value. The problem with reconstruction at critical values arises because in this case the critical values are left critical values by \cite{ZV}.  According to the discussion at the end of this section, groups $\pi_1(\Rips(X,r),\bullet)$ and $\pi_1(\Rips(S,r),\bullet)$, with  $r=c$ being a critical value and $S \subset X$, are isomorphic via $\mu_r^{\pi_1,S}$ if $S$ is dense enough and contains three equidistant points on certain geodesic circles of length $3r$. In such case the density condition alone does not suffice.
\end{Remark}

\begin{Corollary}\label{CorAaa}
 Suppose $X$ is a compact geodesic space, $C$ is a collection of critical values of $\pi_1$-persistence of $X$ via open Rips complexes, $s>0$, and $S\subset X$ is an $s$-dense subset. Define $\mathcal J = (0, \infty)\setminus \bigcup_{c\in C} [c-2s, c)$. Then
$$
\{\pi_1(\Rips(X,r),\bullet)\}_{r\in \mathcal J} \cong \{\pi_1(\Rips(S,r),\bullet)\}_{r\in \mathcal J}.
$$

Analogous statement holds for $H_1(\_; G)$-persistence for any Abelian group $G$.
\end{Corollary}

\begin{Theorem}\label{ThmRealizPers}
[Obtaining Persistence  by finite samples]
Given a pointed compact geodesic space $X$ and $p>0$ there exists a finite subset $S\subset X$ for which maps $\mu^{\pi_1, S}_*$ induce an isomorphism
$$
\{\pi_1(\Rips(X,r),\bullet)\}_{r>p} \cong \{\pi_1(\Rips(S,r),\bullet)\}_{r>p}
$$
 A generic sample $S\subset X$ however does not have such property. 

Analogous statement holds for $H_1(\_; G)$-persistence  for any Abelian group $G$.
\end{Theorem}

\begin{proof} 
Throughout the proof we will be using Theorem \ref{ThmPerBasis} and its  notation without any particular notice.  
In particular, 
we choose geodesic lassos $\{\beta_i=\alpha_i * \gamma_i * \alpha_i^-\}_{i\in J}$, with $\gamma_i$ being a geodesic circle of length $l_i$,  so that
 $\big\{\pi_1(\Rips(X,r),\bullet)\big\}_{r>0}\cong\big\{\L(X,fin, \pi_1)/
\{\beta_i\big\}_{l_i<r}
\}_{r>0}$. Recall that the set of critical values coincides with $\{l_i/3\}_{i\in J}$.

\textbf{Constructing $S$.} Let $2s$ denote the distance between $p$ and the largest critical value smaller than $p$. If such critical value does not exist choose any positive $s< p/2$. Choose a finite $s$-dense subset of $S \subset X$ containing $\bullet$.  
Define a finite (see Theorem \ref{ThmFin} \textbf{a}) collection $C=\{j\in J \mid l_j/3 \geq p\}$. For each $j\in C$ choose three equidistant points $x^j_1, x^j_2, x^j_3$ on $\gamma_j$, with $x^j_1$ being the endpoint of $\alpha_j$, and add all these points to $S$. For a sketch see Figure \ref{FigL3P}. Note that $d(x^j_m, x^j_n)=l_j/3, \forall m \neq n$ as $\gamma_j$ are geodesic circles. By Proposition \ref{PropMapsFinApprox1} $\mu_r^{\pi_1, S}$ is injective, so it only remains to show  that $\mu_r^{\pi_1, S}$ is injective, $\forall r>p$.
 
\begin{figure}
\begin{tikzpicture}[scale=.9]
\draw [thick](-2.8,.9) ..controls (0,.1)..(2.8,.9);
\draw [thick](-2.8,-1.7) ..controls (0,-1)..(2.8,-1.7);
\draw [red](2,0) ..node[below right]{$\alpha_j$} controls (1,-.5) ..(.25,0);
\draw (2,0) node{$\bullet$};
\draw [red, thick](0,-1.2) arc (-90:90:.3 and .75);
\draw [red, thick, dashed](0,.3) arc (90:270:.3 and .75);
\draw [red](0, -1.5) node {$\gamma_j$};
\node [rectangle,  scale=.5, fill=black,draw, label=right:$x_j^1$] at (0.23,0) {};
\node [rectangle,  scale=.5, fill=black,draw, label=right:$x_j^2$] at (0.2,-1) {};
\node [rectangle,  scale=.5, fill=gray,draw, label=left:$x_j^3$] at (-.3,-.45) {};
\draw [ thick](-2.8,-1.7) arc (-70:70:.5 and 1.4);
\draw [ thick](-2.8,.9) arc (110:250:.4 and 1.4);
\draw [ thick](2.8,-1.7) arc (-70:70:.5 and 1.4);
\draw [ thick, dashed](2.8,.9) arc (110:250:.4 and 1.4);
\end{tikzpicture}

\caption{A sketch of lasso $\beta_j= \alpha_j * \gamma_j * \alpha_j^-$ used in the proof of Theorem \ref{ThmRealizPers} and the corresponding choice of three equidistant points $x_j^1, x_j^2,$ and $x_j^3$ on a geodesic circle $\gamma_j$. Path $\alpha_j$ connects the basepoint $\bullet$ and point $x_j^1$ on  $\gamma_j$. The underlying space in this particular case is a tube.}
\label{FigL3P}
\end{figure}
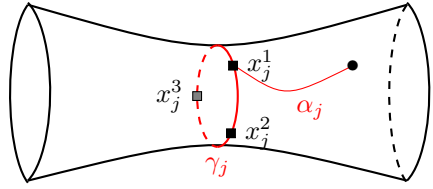
 
\textbf{Homotopy $H$ in $X$.} Fix $r>p$ and choose an $r$-loop $L$ of the form $\bullet = y_0, y_1, \ldots, y_k, y_{k+1}=\bullet$ representing an element of the kernel of $\mu_r^{\pi_1, S}$, i.e., $[L]_{\pi_1}\in \pi_1(\Rips(X,r),\bullet)$ is trivial. Let $\beta_L$ denote a filling of $L$ in $X$, with $\beta^i_L$ representing the corresponding section between $y_i$ and $y_{i+1}$. By assumption we have $[\beta_L]_{\pi_1}\in \L(X, 3r, \pi_1)$. Using Theorem \ref{ThmFin} \textbf{b}  and the fact that $\L(X,*, \pi_1) \leq \pi_1(X, \bullet)$ are normal subgroups we conclude  $\beta_L \simeq \tilde \beta_L * \hat \beta_L$ with $\tilde \beta_L \in \L(X, 3p, \pi_1)$ being a concatenation of lassos generating $\L(X, 3p, \pi_1)$, and $\hat \beta_L$ being a concatenation of geodesic $3 l_{j_{\hat i}}$-lassos $\hat \beta_{j_{\hat i}} = \hat \alpha_{j_{\hat i}} * \hat \gamma_{j_{\hat i}} * \hat \alpha_{j_{\hat i}}^-$ for some ${j_{\hat i}}\in C, l_{j_{\hat i}}< 3r$. In particular, there exists a homotopy $H \colon I \times I \to X, H|_{I \times \{0\}}= \beta_L,  H|_{I \times \{1\}}= \tilde\beta_L * \hat \beta_L$.

\textbf{Homotopy $F$ in $\Rips(S, r)$.} Let $\tilde c$ denote the largest critical value smaller than $r$ and define $\delta = (1/2) \cdot \min \{r-p, r-\tilde c\}$. Let $\mathcal U =\{B(x, \delta/2) \mid x\in X\}$ be an open cover of $X$. Choose a triangulation $\Delta$ of $I \times I$ subordinated to $H^{-1}(\mathcal U)$, i.e., for each simplex of $\Delta$ the diameter of its image by $H$ is less than $\delta$. Further subdivide $\Delta$ so that for each lasso $\alpha * \gamma * \alpha^-$ from the mentioned decomposition of $\tilde \beta_L$ or $\hat \beta_L$:
\begin{description}
 \item [(i)] the points of $I \times I$ corresponding to the endpoints of $\alpha$ and $\alpha^-$ are vertices of $\Delta$;
  \item [(ii)] vertices of $\Delta^{(0)}$ on the domain of $\alpha^-$ precisely backtrack vertices of $\Delta^{(0)}$ on the domain of $\alpha$, meaning that the concatenation of these two $\delta$-samples is a $\delta$-null $\delta$-loop.
\end{description}

For each vertex $v\in \Delta ^{(0)}$ choose $\sigma(v)\in S$ observing the following rules:
\begin{enumerate}
 \item $d(H(v),\sigma(v))< s$;
 \item if $v\in I \times \{0\}$ is a vertex on the domain of $\alpha_L^i$ then $\sigma(v)$ is either $y_i$ or $y_{i+1}$. This implies that $L$ is homotopic in $\Rips(S, r)$ to the $r$-loop defined by $\sigma$ applied to $\Delta^{(0)} \cap I \times \{0\}$;
 \item for each lasso $\alpha * \gamma * \alpha^-$ from the mentioned decomposition of $\tilde \beta_L$ or $\hat \beta_L$, $\sigma$ applied to the vertices on the domain of $\alpha^-$ precisely backtracks $\sigma$ applied to the vertices  on the domain of $\alpha$ (which is possible by (ii)), meaning that the concatenation of these two $s$-samples when mapped by $\sigma$ is $r$-null $r$-loop.
\end{enumerate}
Define now a homotopy $F\colon (I \times I, \Delta)\to \Rips(S, r)$ using the combinatorial structure of triangulation $\Delta$ by mapping $v \mapsto \sigma(v), \forall v\in \Delta^{(0)}$. This rule induces a well defined simplicial map $F$ as 
$$
d(\sigma(v_1), \sigma(v_2))< \delta + 2s < r
$$
for adjacent $v_1, v_2 \in \Delta^{(0)}$.

\textbf{Contract each sample of a lasso:} Take any lasso $\alpha * \gamma * \alpha^-$ from the mentioned decomposition of $\tilde \beta_L$ or $\hat \beta_L$. $F$ restricted to domain of $\gamma$ is contractible in $\Rips(S, r)$ by Proposition \ref{PropSNull}(1) for $\tilde \beta_L$ and Proposition \ref{PropSNull}(2) for $\hat \beta_L$. Remaining $F$ restricted to domain of $\alpha$  and $\alpha^-$ is contractible in $\Rips(S, r)$ by (3) above.
\end{proof}

\begin{Theorem}\label{ThmReconstruct}
 Given a compact geodesic SLSC space $X$, the entire $\pi_1$-persistence of $X$ can be reconstructed from an appropriate finite subset $S\subset X$.
 
Analogous statement holds for $H_1(\_; G)$-persistence of $G$-SLSC spaces for any Abelian group $G$.
\end{Theorem}

\begin{proof}
 If $X$ is simply connected there is nothing to prove. Setting $p$ to be the smallest $\pi_1$-critical value of $X$, use Theorem \ref{ThmRealizPers} to reconstruct  $\{\pi_1(\Rips(X,r),\bullet)\}_{r>p}$. Theorem \ref{ThmPerBasis} implies  $\{\pi_1(\Rips(X,r),\bullet)\}_{r\leq p}\cong \{\pi_1(X, \bullet )\}_{r \leq p}$ with bonding maps being identites. 
 \end{proof}

It is easy to see that Theorem \ref{ThmReconstruct} does not hold for non-compact spaces or non-SLSC spaces. From \cite{AA} we can extract that $\cRips(S^1, 1/3)$ is homotopy equivalent to the uncountable wedge of two-dimensional spheres, where $S^1$ is a geodesic circle of circumference $1$.
Hence neither  an analogous statement of Theorem \ref{ThmReconstruct} nor that of Theorem \ref{ThmRealizRips} can hold for higher-dimensional persistence via closed complexes.

Given a fixed $X$, there may be various representations of persistence  provided by Theorem \ref{ThmPerBasis}. Such representations may use different geodesic circles $\gamma_i$ if, for example, some such geodesic circle $\gamma_j$ is freely homotopic to a different geodesic circle of the same circumference. 
It is not hard to see that a sufficiently dense subset  satisfies conclusions of Theorem \ref{ThmReconstruct} if and only if it contains three equidistant points on each geodesic circle $\gamma_i$ arising from some representation of persistence as described by Theorem \ref{ThmPerBasis}. 
This implies that a random sufficiently dense sample can't be expected to satisfy the conclusions of Theorem \ref{ThmReconstruct}, unless the underlying probability distribution generates three equidistant points on the corresponding geodesic circles with non-negative probability.

\section{Order on approximating persistences}

In this section we introduce a partial order on the set of all  persistences induced by subsets of $X$. Using this order we may formalize the following statement: persistence of $X$ is the minimum of persistences arising from finite samples.
The this statement shows that when approximating persistence of $X$ by finite subsets, the best approximation is the minimal one (instead of some sort of average). The intuitive explanation is the following. Critical values of persistence of $X$ arise from geodesic circles. The corresponding critical values appearing in persistence of a finite sample $S$ arise from corresponding approximations of geodesic circles using points of $S$, which are longer than the original geodesic circle. This gives us further clues for computational implementation of such persistence: when improving the sample, it is enough to increase the density only around critical circles.


\begin{Definition}\label{DefPers}
 Suppose $G$ is a group and $r_0>0$. An \textbf{initially constant surjective persistent group} starting with $G$ at $r_0$ is a persistent group consisting of a collection of groups  $\mathcal H = \{G_r\}_{r>r_0}$ along with surjective commuting bonding homomorphisms $i_{p,q}\colon G_p \to G_q$ for all $p<q$, with an additional property that there exists $\eps_{\mathcal H}>r_0$, so that for all $p<q<\eps_{\mathcal H}$ maps $i_{p,q}$ are identities on $G$. 
 
 A \textbf{homomorphism} $\phi \colon \mathcal H \to  {\mathcal H'}$ between two such persistences (with the same parameter bound) $\mathcal H = \{G_r\}_{r>r_0}$ and $\mathcal H' = \{G'_r\}_{r>r_0}$ is a collection of homomorphisms $\phi_r \colon G_r \to G'_r$, which commute with the bonding maps and are identities for $r$ close enough to $r_0$ (this implies that all $\phi_r$ are surjective).
Such a homomorphism is an \textbf{isomorphism} if all $\phi_i$ are isomorphisms.
 \end{Definition}

\begin{Definition}
 Suppose $X$ is a compact geodesic space and $c$ is the smallest critical value of $\pi_1$-persistence via open Rips complexes. Choose $s< c/2$. Define a collection of initially constant surjective persistences starting with $\pi_1(X,\bullet)$: 
 $$
 FP(X, s)= \Big\{ \{\pi_1(\Rips(A, r), \bullet)\}_{r>2s} \mid A \quad s\textit{-dense finite subset of } X \Big\}.
 $$
 The corresponding bonding maps are induced by the standard inclusions attached to the Rips filtration. The facts that $c>0$ exists and that defined persistences are indeed initially constant at $\pi_1(X,\bullet)$ follow from the results of \cite{ZV}.
\end{Definition}
 

\begin{Definition}
 \label{DefPO}
Let $G$ be a group and $r_0>0$. Suppose $\mathcal H= \{G_r\}_{r>r_0}$ and $ {\mathcal H'}= \{G'_r\}_{r>r_0}$ are initially constant surjective persistences starting with $G$. Define an order:
$\mathcal H \leq  {\mathcal H'} $ iff there exists a homomorphism $\phi \colon \mathcal H' \to  {\mathcal H}$. It is easy to see that relation $\leq$ is a partial order on the set of initially constant surjective persistences starting with $G$ at $r_0$.
\end{Definition}

If $A\subset B$ are $s$-dense subsets of $X$, where $c>2s$ is the smallest critical value of $\pi_1$-persistence of $X$ via open Rips complexes, then 
$\{\pi_1(\Rips(A, r), \bullet)\}_{r>2s} \geq \{\pi_1(\Rips(B, r), \bullet)\}_{r>2s}$. Since smaller elements in our partial order represent better approximations of persistence, the last statement formalises the heuristic that larger sets provide a better approximation. Also, it shows that for a collection of $s$-dense subsets of $X$, their union is the lower bound in our poset.

\begin{Prop}
 Let $G$ be a group and $r_0>0$. Suppose $\mathcal H= \{G_r\}_{r>r_0}$ and $ {\mathcal H'}= \{G'_r\}_{r>r_0}$ are initially constant surjective persistences starting with $G$, whose bonding maps are $i_{*,*}$ and $i'_{*,*}$ respectively. Assume $\eps=\eps_{\mathcal H'}=\eps_{\mathcal H}$ and choose $p\in (r_0, \eps)$. Then the relationship $\mathcal H \leq  {\mathcal H'} $ is equivalent to the following condition: $\ker i'_{p,q} \leq \ker i_{p,q} \leq G, \forall q>p$.  
 \end{Prop}

\begin{proof}
Since persistences are initially surjective, the choice of $p$ does not play any role. If $\mathcal H \leq  {\mathcal H'} $, then $\ker i'_{p,q} \leq \ker i_{p,q} \leq G, \forall q>p$ by the commutativity of a homomorphism $\phi \colon \mathcal H' \to  {\mathcal H}$ with the bonding maps.

If $\ker i'_{p,q} \leq \ker i_{p,q}, \forall q>p$, then $\phi$ consists of well defined homomorphisms $f_q\colon G'_q \to G_q$ determined by the following conditions:
\begin{itemize}
\item  $f_r$ is the identity $\forall r\leq p$;
 \item $f_q(i'_{p,q}(x))=f_q(i_{p,q}(x)), \forall x\in G$.
\end{itemize}
Note that the  kernel of $f_q$ is $\ker i_{p,q} / \ker i'_{p,q}$.
\end{proof}

\begin{Rem}
While the definitions of this section refer to initially constant surjective persistence groups, it is clear that the same definitions can be used to define initially constant surjective persistence modules and a partial order therein.

Suppose  $\mathcal H$ and  ${\mathcal H'}$ are initially constant surjective persistence modules starting with a finite dimensional vector space $G$ at $r_0>0$. It is clear that each of the intervals of the barcode of $\mathcal H$ (and of ${\mathcal H'}$) is of the form $(r_0, d_i)$ or $(r_0, d_i]$. For any $p\in (r_0, \eps_\mathcal{H})$ and $q>p$, the number of intervals of the barcode of $\mathcal H$ not containing $q$ is $\dim \ker i_{p,q}$. We conclude the following: if $\mathcal H \leq  {\mathcal H'} $, then the intervals in the barcode of $\mathcal H$ are shorter or equal than those of $\mathcal H'$ (when both sets of intervals are sorted by their length, i.e., the longest interval of ${\mathcal H'}$ is longer than or equal to the longest interval of ${\mathcal H}$, etc.). It is easy to see that this implication is not reversible.
\end{Rem}

The following theorem explains how the entire $\pi_1$-persistence of (usually uncountable space) $X$ is the minimal persistence, obtained by finite samples.

\begin{Theorem}
 \label{ThmSamplingMin}
 [Sampling minimal property]
 Suppose $X$ is a compact geodesic SLSC space, with $c>0$ being the smallest critical value of $\pi_1$-persistence of $X$ via open Rips complexes. Choose $s< c/2$. Then $\{\pi_1(\Rips(X, r), \bullet)\}_{r>2s}$ (recall that it completely describes $\{\pi_1(\Rips(X, r), \bullet)\}_{r>0}$) is isomorphic to the minimal element in $FP(X,s)$. 
\end{Theorem}

\begin{proof}
$\mathcal{H}=\{\pi_1(\Rips(X, r), \bullet)\}_{r>2s}$ is smaller than or equal to  any element of $FP(X,s)$ by the paragraph following Definition \ref{DefPO}.

 By  Theorem \ref{ThmRealizPers} there exists an element in $FP(X,s)$, which is isomorphic to $\mathcal{H}$, hence it is formally smaller or equal to $\mathcal{H}$. We conclude that this element is  the minimal element of $FP(X,s)$. 
 \end{proof}

\begin{Remark}
 By results of \cite{ZV} the statement of Theorem \ref{ThmSamplingMin} also holds for closed Rips filtrations. Furthermore, for any Abelian group $G$, the whole of this section could be stated for $H_1(\_; G)$-persistence and compact geodesic $G$-SLSC spaces. In our future work we plan to expand this result to other settings when possible, i.e., to higher dimensions, closed filtrations, etc.
\end{Remark}

\section{Concluding remarks and future work}

In this paper we showed that even though a one-dimensional persistence of a geodesic space is defined by Rips complexes on uncountably many points, the persistence itself is nicely approximated and extractable from finite samples. These results along with other observation on the geometry of critical values described in \cite{ZV} form a basis for the upcoming work on computational approximation of such persistence. On the other hand we are developing stronger finiteness results for the $\pi_1$-persistence in such cases. 

Analogous results mostly hold for closed Rips filtrations (again, see \cite{ZV} for the relationship). The exceptions are Theorem \ref{ThmRealizRips} and Corollary \ref{CorAaa} where caution is needed at critical values as described in Remark \ref{RemException}.
A corresponding theory could also be developed for filtrations induced by \v Cech complexes.
In the future we plan to look at analogous results in higher dimensions, which would assist with the computation of higher-dimensional persistence. However, it is clear  that such results are not always obtainable. Results of \cite{AA, AAS} imply that finiteness results can not hold for higher-dimensional persistences via closed filtrations. It is also apparent that compactness is required for our finiteness results. In the absence of a geodesic metric even the fundamental groups of Rips complexes of very simple spaces (such as $\{0,1\}\times [0,1]\subset \RR^2$ for closed Rips filtration) may be uncountable \cite{Cha2}.


\end{document}